\documentclass[11pt]{article}
\usepackage{amsmath}
\usepackage{amsthm}
\usepackage{amssymb,amsfonts}
\usepackage{fancyhdr}
\usepackage{hyperref}
\usepackage{latexsym,microtype}
\usepackage{todonotes}
\usepackage{nicefrac}
\usepackage{epsfig,euscript}
\usepackage{stmaryrd}
\usepackage{setspace,mathrsfs}
\usepackage{enumerate}
\usepackage[all]{xypic}
\usepackage{bbm,ifpdf,tikz}
\ifpdf
\usepackage{pdfsync}
\fi

\newtheorem{theorem}{Theorem}[section]

{
\theoremstyle{definition}

\newtheorem{example}[theorem]{Example}

\newtheorem{remark}[theorem]{Remark}
}

\newenvironment{proofqniform}{\noindent {\bf Proof of Theorem \ref{q-niform}:}}{\qed \par}

\newcommand{\excise}[1]{}

\renewcommand{\dim}{\operatorname{dim}}

\newcommand{\rk}{\operatorname{rk}}
\newcommand{\crk}{\operatorname{crk}}

\renewcommand{\and}{\qquad\text{and}\qquad}
\newcommand{\Ind}{\operatorname{Ind}}

\newcommand{\triv}{\operatorname{triv}}

\newcommand{\Z}{\mathbb{Z}}

\newcommand{\C}{\mathbb{C}}

\newcommand{\Fq}{\mathbb{F}_q}
\newcommand{\Fqb}{\overline{\mathbb{F}}_q}

\newcommand{\Ql}{\mathbb{Q}_\ell}

\newcommand{\la}{\lambda}

\newcommand{\OS}{OS}

\newcommand{\gm}{\mathbb{G}_m}

\newcommand{\Gn}{\operatorname{GL}_n(q)}
\newcommand{\Gk}{\operatorname{GL}_k(q)}
\newcommand{\Gnk}{\operatorname{GL}_{n-k}(q)}
\newcommand{\Bn}{\operatorname{B}_n(q)}
\newcommand{\Bk}{\operatorname{B}_k(q)}
\newcommand{\Bnk}{\operatorname{B}_{n-k}(q)}
\newcommand{\bOS}{\overline{\OS}}
\renewcommand{\P}{\mathbb{P}}
\newcommand{\Gm}{\mathbb{G}_m}

\begin{document}
\spacing{1.2}
\noindent{\Large\bf Equivariant Kazhdan-Lusztig polynomials of \boldmath{$q$}-niform matroids}\\

\noindent{\bf Nicholas Proudfoot}\\
Department of Mathematics, University of Oregon,
Eugene, OR 97403\\
njp@uoregon.edu\\

{\small
\begin{quote}
\noindent {\em Abstract.} 
We study $q$-analogues of uniform matroids, which we call $q$-niform matroids.
While uniform matroids admit actions of symmetric groups, $q$-niform matroids admit actions
of finite general linear groups.  We show that the equivariant Kazhdan-Lusztig polynomial
of a $q$-niform matroid is the unipotent $q$-analogue of the equivariant Kazhdan-Lusztig polynomial
of the corresponding uniform matroid, thus providing evidence for the positivity conjecture for equivariant
Kazhdan-Lusztig polynomials.
\end{quote} }

\section{Introduction}
For any matroid $M$, the {\bf Kazhdan-Lusztig polynomial} $P_M(t) \in \Z[t]$ was introduced in \cite{EPW}.
In the case where the matroid $M$ admits the action of a finite group $W$, one can define the {\bf equivariant Kazhdan-Lusztig polynomial}
$P_M^W(t)$ \cite{GPY}; this is a polynomial whose coefficients are virtual representations of $W$ (in characteristic zero)
with dimensions equal to the coefficients of $P_M(t)$.

Though these polynomials admit elementary recursive definitions, there are not many families of matroids for which explicit formulas
are known.  Non-equivariant formulas exist for thagomizer matroids \cite{thag} and fan, wheel, and whirl matroids \cite{fan-wheel-whirl}.
Kazhdan-Lusztig polynomials of braid matroids have been studied extensively, both in the equivariant \cite{fs-braid}
and non-equivariant \cite{Karn-Wakefield} settings, though no simple formulas have been obtained.

The most interesting explicit formulas that we have are for uniform matroids.  Let $U_{n,m}$ be the uniform matroid
of rank $n-m$ on a set of cardinality $n$, which admits an action of the symmetric group $S_n$.  
For any partition $\la$ of $n$, let $V[\la]$ be the associated irreducible representation of $S_n$.
The following theorem was
proved in \cite[Theorem 3.1]{GPY}; an independent proof of the non-equivariant statement was later given in 
\cite[Theorem 1.2]{Chinese-uniform}.

\begin{theorem}\label{uniform}
Let $C_{n,m}^i$ be the coefficient of $t^i$ in the $S_n$-equivariant Kazhdan-Lusztig polynomial of $U_{n,m}$,
and let $c_{n,m}^i := \dim C_{n,m}^i$ be the corresponding non-equivariant coefficient.
\begin{itemize}
\item $C_{n,m}^0 = V[n]$, and for all $i>0$, $$C_{n,m}^i = \sum_{b=1}^{\min(m,n-m-2i)} V[n-2i-b+1,b+1,2^{i-1}].$$
\item $c_{n,m}^0 = 1$, and for all $i>0$, $$c_{n,m}^i = 
\sum_{b=1}^{\min(m,n-m-2i)} \frac{(n - 2i - 2b + 1)n!}{(n-i-b)(n-i-b+1)(i + b)(i + b - 1)(n-2i-b)!(b - 1)!i!(i - 1)!}.$$
\end{itemize}
\end{theorem}

The purpose of this note is to obtain a $q$-analogue of Theorem \ref{uniform}.
Let $q$ be a prime power, and let $U_{n,0}(q)$ be the rank $n$ matroid associated with
the collection of all hyperplanes in the vector space $\Fq^n$, which we regard as the $q$-analogue
of the Boolean matroid of rank $n$.  For any natural number $m\leq n$,
let $U_{n,m}(q)$ be the truncation of $U_{n,0}(q)$ to rank $n-m$.
More concretely, a basis for $U_{n,m}(q)$ is a set of $n-m$ hyperplanes whose intersection has dimension $m$.
The matroid $U_{n,m}(q)$ is the $q$-analogue of the uniform matroid $U_{n,m}$,
and we will therefore refer to it as a {\bf \boldmath{$q$}-niform matroid}.
This matroid was also studied in \cite{HRS}, where the authors computed the Hilbert series of its Chow ring.
The $q$-niform matroid $U_{n,m}(q)$ admits a natural action of the group $\Gn$ of invertible $n\times n$
matrices with coefficients in $\Fq$, which is the $q$-analogue of $S_n$.

The representation theory of $\Gn$ is much more complicated than the representation theory of $S_n$.
However, there is a certain subset of irreducible representations of $\Gn$, known as {\bf irreducible unipotent representations},
that correspond bijectively to the irreducible representations of $S_n$.  
For any partition $\la$ of $n$, let $V(q)[\la]$ be the associated irreducible unipotent representation of $\Gn$.
For any positive integer $k$, we use the standard notation 
$$[k]_q := 1 + q + \cdots + q^{k-1}\and [k]_q! := [k]_q[k-1]_q\cdots[1]_q.$$
The following theorem, which is our main result, says that the equivariant Kazhdan-Lusztig coefficients of $U_{n,m}(q)$
are precisely the unipotent $q$-analogues of the equivariant Kazhdan-Lusztig coefficients of $U_{n,m}$.

\begin{theorem}\label{q-niform}
Let $C_{n,m}^i(q)$ be the coefficient of $t^i$ in the $\Gn$-equivariant Kazhdan-Lusztig polynomial of $U_{n,m}(q)$,
and let $c_{n,m}^i(q) := \dim C_{n,m}^i(q)$ be the corresponding non-equivariant coefficient.
\begin{itemize}
\item $C_{n,m}^0(q) = V(q)[n]$, and for all $i>0$, $$C_{n,m}^i(q) = \sum_{b=1}^{\min(m,n-m-2i)} V(q)[n-2i-b+1,b+1,2^{i-1}].$$
\item $c_{n,m}^0(q) = 1$, and for all $i>0$, $c_{n,m}^i(q)$ is equal to
$$\sum_{b=1}^{\min(m,n-m-2i)} 
\frac{q^{b-1+i(i+1)}\;[n - 2i - 2b + 1]_q[n]_q!}{[n-i-b]_q[n-i-b+1]_q[i + b]_q[i + b - 1]_q[n-2i-b]_q![b - 1]_q![i]_q![i - 1]_q!}.$$
\end{itemize}
\end{theorem}

\begin{remark}
For any matroid $M$, the coefficients of $P_M(t)$ are conjectured to be non-negative \cite[Conjecture 2.3]{EPW}.
More generally, the coefficients of $P_M^W(t)$ are conjectured to be honest (rather than virtual) representations of $W$
\cite[Conjecture 2.13]{GPY}.  These conjectures are proved when $M$ is realizable \cite[Theorem 3.10]{EPW} 
(respectively equivariantly realizable \cite[Corollary 2.12]{GPY}), but no proof exists in the general case.
The matroid $U_{n,m}$ is always realizable, but it is not equivariantly realizable unless $m\in\{0,1,n-1,n\}$
(of these, only the $m=1$ case yields nontrivial Kazhdan-Lusztig coefficients).
Similarly, the matroid $U_{n,m}(q)$ is always realizable, but it is typically not equivariantly realizable.
Thus Theorems \ref{uniform} and \ref{q-niform} both provide significant evidence for the equivariant
non-negativity conjecture.
\end{remark}

\begin{remark}
Theorem \ref{uniform} implies that $\{C_{n,m}^i \mid n\geq m\}$ admits the structure of a finitely
generated FI-module \cite[Theorem 1.13]{CEF}, 
while Theorem \ref{q-niform} implies that $\{C_{n,m}^i(q) \mid n\geq m\}$ admits the structure of a finitely
generated VI-module \cite[Theorem 1.6]{GanW}.  
In order to define these structures in a natural way, we would need need to be able to define $C_{n,m}^i$
and $C_{n,m}^i(q)$ as actual vector spaces rather than as isomorphism classes of vector spaces.
The matroid $U_{n,1}$ is equivariantly realizable, which means that we have a cohomological interpretation of $C_{n,1}^i$, 
and we obtain a canonical $\operatorname{FI^{op}}$-module structure from \cite[Theorem 3.3(1)]{fs-braid};
dualizing then gives a canonical finitely generated FI-module.
In joint work with Braden, Huh, Matherne, and Wang, the author is working to construct a canonical vector space 
isomorphic to the coefficient of $t^i$ in $P_M(t)$ for any matroid $M$.  When this goal is achieved, we believe
that this construction will induce a canonical $\operatorname{FI^{op}}$-module structure on $\{C_{n,m}^i \mid n\geq m\}$
and a canonical $\operatorname{VI^{op}}$-module structure on $\{C_{n,m}^i(q) \mid n\geq m\}$, each with finitely generated duals.
\end{remark}

Our proof of Theorem \ref{q-niform} relies heavily on Theorem \ref{uniform} along with the Comparison Theorem
(Theorem \ref{properties}), which roughly says that calculations involving Harish-Chandra induction of 
unipotent representations of finite general linear groups are essentially equivalent to the analogous calculations 
for symmetric groups.  
The only additional ingredients in the proof are to check that the Orlik-Solomon algebra
of $U_{n,m}(q)$  is the unipotent $q$-analogue
of the Orlik-Solomon algebra of $U_{n,m}$ (Example \ref{truncation})
and that the recursive formula for $C_{n,m}^i(q)$ is essentially the same as the recursive formula for $C_{n,m}^i$
(Equations \eqref{uniform-recursion} and \eqref{q-niform-recursion}).

\vspace{\baselineskip}
\noindent
{\em Acknowledgments:}
The author is indebted to June Huh for help with formulating the main result and to Olivier Dudas for 
help with proving it.  The author is supported by NSF grant DMS-1565036.  

\section{Unipotent representations and the Comparison Theorem}\label{sec:comparison}
Given a pair of natural numbers $k\leq n$ and a pair of representations $V$ of $S_k$ and $V'$ of $S_{n-k}$,
we define
$$V*V' := \Ind_{S_k\times S_{n-k}}^{S_n}\Big(V\boxtimes V'\Big).$$
Irreducible representations of the symmetric group $S_n$ are classified by partitions of $n$.
Given a partition $\la$, let $V[\la]$ be the associated representation.  
For each cell $(i,j)$ in the Young diagram for $\la$, let $h_\la(i,j)$ be the corresponding hook length;
then the dimension of $V[\la]$ is equal to $$\frac{n!}{\prod h_\la(i,j)}.$$

We now review some analogous statements and constructions in the representation theory of finite general linear groups.
Given a pair of natural numbers $k\leq n$, let $P_{k,n}(q)\subset\Gn$ denote 
the parabolic subgroup associated with the Levi $\Gk\times\Gnk$.
Given a pair of representations $V(q)$ of $\Gk$ and $V'(q)$ of $\Gnk$,
we obtain a representation $V(q)\boxtimes V'(q)$ of $\Gk\times\Gnk$, and we may interpret this as a representation of $P_{k,n}(q)$
via the natural surjection $P_{k,n}(q)\to \Gk\times\Gnk$.  We then define
$$V(q)*V'(q) := \Ind_{P_{k,n}(q)}^{\Gn}\Big(V(q)\boxtimes V'(q)\Big).$$
This operation is called {\bf Harish-Chandra induction}.

Let $\Bn\subset \Gn$ be the subgroup of upper triangular matrices.  An irreducible representation
of $\Gn$ is called {\bf unipotent} if it appears as a direct summand of the representation
$$\C\big[\Gn/\Bn\big] =
\Ind_{\Bn}^{\Gn}\!\big(\triv_{\Gn}\big).$$
An arbitrary representation is called unipotent if it is isomorphic to a direct sum of irreducible unipotent representations.

\begin{theorem}\label{properties}
Let $q$ be a prime power and $n$ a natural number.
\begin{enumerate}
\item Irreducible unipotent representations of $\Gn$ are in canonical bijection with partitions of $n$.
\item The irreducible unipotent representation $V(q)[\la]$ associated with the partition $\la$ has dimension
$$q^{\sum (k-1)\la_k}\frac{[n]_q!}{\prod [h_\la(i,j)]_q}.$$
\item If $k\leq n$, $V(q)$ is a unipotent representation of $\Gk$, and $V'(q)$ is a unipotent representation of $\Gnk$,
then $V(q)*V'(q)$ is a unipotent representation of $\Gn$.
\item Let $\la$, $\mu$, and $\nu$ be partitions of $n$, $k$, and $n-k$, respectively.
The multiplicity of $V(q)[\la]$ in $V(q)[\mu]*V(q)[\nu]$ is equal to the multiplicity of $V[\la]$ in $V[\mu] * V[\nu]$.
\end{enumerate}
\end{theorem}

\begin{proof}
Statements 1 and 4 appear in \cite[Theorem B]{Curtis75}.
The fact that the dimension of $V(q)[\la]$ is polynomial in $q$ appears in \cite[Theorem 2.6]{BensonCurtis}.
For an explicit calculation of this polynomial, see \cite[Equation (1.1)]{DipperJames}.
Finally, Statement 3 follows from the fact that $\C\big[\Gk/\Bk\big] * \C\big[\Gnk/\Bnk\big] \cong \C\big[\Gn/\Bn\big]$.
\end{proof}

\begin{remark}
The standard proof of Theorem \ref{properties}(1) is very far from constructive.
One proves that the endomorphism algebra of $\C\big[\Gn/\Bn\big]$ is isomorphic to the Hecke algebra
of $S_n$; this implies that the irreducible constituents of $\C\big[\Gn/\Bn\big]$ are in canonical bijection with irreducible
modules over the Hecke algebra, which are in turn in canonical bijection with irreducible representations of $S_n$.
However, a recent paper of Andrews \cite{Andrews} gives a construction of $V(q)[\la]$ modeled on tableaux, which
is analogous to the usual construction of $V[\la]$.
\end{remark}

\begin{remark}
A generalization of Statement 4 due to Howlett and Lehrer \cite[Theorem 5.9]{Howlett-Lehrer} is commonly referred to  
as the Comparison Theorem.  For the purposes of this paper, we will use this terminology to refer to the entirety of Theorem \ref{properties}.
\end{remark}

\section{Orlik-Solomon algebras}
For any matroid $M$ on the ground set $E$, let $\OS_{\! M}^*$ be the {\bf Orlik-Solomon algebra} of $M$, and let
$$\chi_M(t) := \sum_{i=0}^{\rk M} (-1)^i \dim \OS_{\! M}^{i} t^{\rk M - i}$$
be the {\bf characteristic polynomial} of $M$.  The Orlik-Solomon algebra is a quotient of the exterior algebra
over the complex numbers with generators $\{x_e\mid e\in E\}$.  Let $\bOS_{\! M}^*$ be the {\bf reduced Orlik-Solomon algebra} of $M$,
which is defined as the subalgebra of $\OS_{\! M}^*$ generated by $\{x_e -x_{e'}\mid e,e'\in E\}$.
If $\rk M > 0$, then we have a graded algebra isomorphism
\begin{equation}\label{kunneth}\OS_{\! M}^* \cong \bOS_{\! M}^* \otimes \C[x]/\langle x^2\rangle\end{equation}
and therefore a vector space isomorphism 
\begin{equation}\label{kunneth individual}\OS_{\! M}^i \cong \bOS_{\! M}^i \oplus \bOS_{\! M}^{i-1}.\end{equation}
If a finite group $W$ acts on $M$, we obtain induced actions on $\OS^*_{\! M}$ and $\bOS^*_{\! M}$, and the isomorphisms
of Equations \eqref{kunneth} and \eqref{kunneth individual} are $W$-equivariant.

\begin{example}\label{linear}
Suppose that $V$ is a vector space over $\Fq$, and that $\{H_e\mid e\in E\}$ is a collection of hyperplanes
with associated matroid $M$.  Fix a prime $\ell$ that does not divide $q$, and fix an embedding of $\Ql$ into $\C$.
Let $$X := V(\Fqb)\smallsetminus \bigcup_{e\in E} H_e(\Fqb)
\and \P X := \P V(\Fqb)\smallsetminus \bigcup_{e\in E} \P H_e(\Fqb).$$
Then we have canonical isomorphisms
$$\OS_{\! M}^* \cong H^*(X; \Ql)\otimes_{\Ql}\C
\and 
\bOS_{\! M}^* \cong H^*(\P X; \Ql)\otimes_{\Ql}\C,$$
where the cohomology rings are $\ell$-adic \'etale cohomology.
If $\rk M > 0$, then $X\cong \P X \times\gm(\Fqb)$, and Equation \eqref{kunneth} is simply the Kunneth formula.
If $W$ acts on $V$ by linear automorphisms preserving the collection of hyperplanes, we obtain an induced action on $M$,
and these isomorphisms are $W$-equivariant.
\end{example}

\begin{example}
The Boolean matroid $U_{n,0}$ is $S_n$-equivariantly realized by the coordinate hyperplanes in $\Fq^n$.
Its Orlik-Solomon algebra $\OS_{n,0}^*$ is equal to the exterior algebra on $n$ generators, which is isomorphic
to the cohomology of $X_{n,0} \cong \Gm^n(\Fq)$.
As a representation of $S_n$, we have $$\OS_{n,0}^* \cong \Lambda^*\Big( V[n-1,1] \oplus V[n]\Big)
\and \bOS_{n,0}^* \cong \Lambda^*\Big( V[n-1,1]\Big),$$
In particular, this implies that
\begin{equation}\label{reduced uniform}\bOS_{n,0}^i \cong V[n-i,1^i]\end{equation}
for all $i<n$.
\end{example}

\begin{example}
The matroid $U_{n,0}(q)$ is (by definition) $\Gn$-equivariantly realized by the collection of all hyperplanes in $\Fq^n$.
The variety $\P X_{n,0}(q)$ is an example of a Deligne-Lusztig variety for the group $\Gn$.
The techniques developed by Lusztig \cite{Lusztig-Coxeter}
imply that the action of $\Gn$ on the cohomology group of $\P X_{n,0}(q)$ is
given by the unipotent $q$-analogue of Equation \eqref{reduced uniform}:
\begin{equation}\label{reduced q-niform}
\bOS_{n,0}^i(q) \cong V(q)[n-i,1^i]\end{equation}
for all $i<n$.  See \cite[Examples 6.1 and 6.4]{Dudas-EMSS} for a concise and explicit statement of this result.
\end{example}

\begin{example}\label{truncation}
Let $M$ be any matroid, let $d\leq \rk M$ be a natural number, and let $M'$ be the truncation of $M$ to rank $d$.
Then $\OS^*_{\! M'}$ is the truncation of $\OS^*_{\! M}$ to degree $d-1$.  That is, we have a canonical isomorphism
$\bOS_{\! M'}^i \cong \bOS_{\! M}^i$ for all $i\leq d-1$, and $\bOS_{\! M'}^i=0$ for all $i\geq d$.
In the case of Example \ref{linear}, this reflects the fact that 
$\P X'$ is a generic hyperplane section of $\P X$.
In particular, we have
\begin{equation}\label{reduced both}\bOS_{n,m}^i \cong V[n-i,1^i]
\and \bOS_{n,m}^i(q) \cong V(q)[n-i,1^i]\end{equation}
when $i<n-m$, and both groups are zero otherwise.
\end{example}

\section{Kazhdan-Lusztig polynomials}
Let $M$ be a matroid on the ground set $E$ with lattice of flats $L$.  For any $F\in L$, let $M_F$ denote the 
{\bf localization} of $M$ at $F$; this is the matroid on the ground set $F$ whose bases are maximal independent sets of $F$.
Let $M^F$ denote the {\bf contraction} of $M$ at $F$.
If $B$ is a basis for $M_F$, then $M^F$ is obtained from $M$ by 
contracting each element of $B$ and deleting each element of $F\smallsetminus B$.
Equivalently, $M^F$ is a matroid on the ground set $E\smallsetminus F$, and $B'\subset E\smallsetminus F$ is a basis for $M^F$ 
if and only if $B'\cup B$ is a basis for $M$.

\begin{example}
If $F$ is equal to the ground set of $M$ (the maximal flat), then $M_F = M$ and $M^F$ is the matroid of rank zero on the emptyset.
\end{example}

\begin{example}\label{uniform flats}
Proper (that is, non-maximal) flats of $U_{n,m}$ are subsets of $[n]$ of cardinality less than $n-m$.
For such an $F$, $(U_{n,m})_F \cong U_{|F|,0}$ is Boolean, while $U_{n,m}^F \cong U_{n-|F|,m}$.
\end{example}

\begin{example}\label{q-niform flats}
Proper flats of $U_{n,m}(q)$ are collections of linearly independent hyperplanes in $\Fq^n$ of cardinality less than $n-m$.
For such an $F$, $U_{n,m}(q)_F \cong U_{|F|,0}(q)$, while $U_{n,m}(q)^F \cong U_{n-|F|,m}(q)$.
\end{example}

The Kazhdan-Lusztig polynomial of $M$ is characterized by the following three conditions \cite[Theorem 2.2]{EPW}:
\begin{enumerate}
\item If $\rk M = 0$, then $P_M(t) = 1$.
\item If $\rk M > 0$, then $\deg P_M(t) < \tfrac{1}{2}\rk M$.
\item For every $M$, $\displaystyle t^{\rk M} P_M(t^{-1}) = \sum_{F}\chi_{M_F}(t) P_{M^F}(t).$
\end{enumerate}
If $M$ admits the action of a finite group $W$,
the equivariant Kazhdan-Lusztig polynomial is defined by the three analogous conditions, with the 
coefficients of the characteristic polynomial replaced by the graded pieces of the Orlik-Solomon algebra (with corresponding signs),
which are now virtual representations of $W$ rather than integers.  
For every flat $F\in L$, let $W_F\subset W$ denote the stabilizer of $F$.
If $C_{M,W}^i$ is the coefficient of $t^i$ in the $W$-equivariant
Kazhdan-Lusztig polynomial of $M$ and $i < \rk M/2$, we have the following explicit recursive formula \cite[Proposition 2.9]{GPY}:
\begin{equation}\label{ekl-recursion}
C^i_{M,W} = \sum_{\substack{[F]\in L/W\\ 0\leq j\leq \rk F}} (-1)^j\; \Ind_{W_F}^W\!\Big(\OS^{j}_{M_F}
\otimes\; C^{\crk F - i + j}_{M^F,W_F}\Big),\end{equation}
where we take in the sum one flat from each $W$-orbit in $L$.

\begin{example}
Consider the case of the uniform matroid $U_{n,m}$.
Proper flats are subsets of $[n]$ of cardinality less than $n-m$, and the $S_n$-orbit of a flat is determined
by its cardinality.  The stabilizer of a flat of cardinality $k$ is isomorphic to the Young subgroup $S_k\times S_{n-k}\subset S_n$.
Thus Equation \eqref{ekl-recursion} transforms into the following recursion:
\begin{align}\label{uniform-recursion}
\nonumber C^i_{n,m} &= (-1)^i \OS^{i}_{n-m} + \sum_{k=0}^{n-m-1}\sum_{j=0}^k (-1)^j\; \Ind_{S_k\times S_{n-k}}^{S_n}\!\Big(\OS^{j}_{k,0}
\otimes\; C^{n - m - k - i + j}_{n-k,m}\Big)\\
&= (-1)^i \OS^{i}_{n-m} + \sum_{k=0}^{n-m-1}\sum_{j=0}^k (-1)^j\; \OS^{j}_{k,0} * C^{n - m - k - i + j}_{n-k,m},\end{align}
where the first term corresponds to the maximal flat $F = [n]$.
\end{example}

\begin{example}
Consider the case of the $q$-uniform matroid $U_{n,m}(q)$.
Proper flats are collections of linearly independent hyperplanes in $\Fq^n$ of cardinality less than $n-m$, 
and the $\Gn$-orbit of a flat is determined
by its cardinality.  The stabilizer of a flat of cardinality $k$ is isomorphic to the parabolic subgroup $P_{n,k}(q)\subset \Gn$.
Thus Equation \eqref{ekl-recursion} transforms into the $q$-analogue of Equation \eqref{uniform-recursion}:
\begin{align}\label{q-niform-recursion}
\nonumber C^i_{n,m}(q) &= (-1)^i \OS^{i}_{n-m}(q) + \sum_{k=0}^{n-m-1}\sum_{j=0}^k (-1)^j\; \Ind_{P_{n,k}(q)}^{\Gn}\!\Big(\OS^{j}_{k,0}(q)
\otimes\; C^{n - m - k - i + j}_{n-k,m}(q)\Big)\\
&= (-1)^i \OS^{i}_{n-m}(q) + \sum_{k=0}^{n-m-1}\sum_{j=0}^k (-1)^j\; \OS^{j}_{k,0}(q) * C^{n - m - k - i + j}_{n-k,m}(q).\end{align}
\end{example}

\begin{proofqniform}
By Equations \eqref{kunneth individual}, \eqref{reduced uniform}, and \eqref{reduced q-niform}, Equation \eqref{q-niform-recursion}
is precisely the unipotent $q$-analogue of Equation \eqref{uniform-recursion}.  
Then by Theorem \ref{properties}, the first part of Theorem \ref{q-niform} is equivalent to the first part
of Theorem \ref{uniform}.  The second part of Theorem \ref{q-niform} follows from Theorem \ref{properties}(2).
\end{proofqniform}

\bibliography{./symplectic}

\newcommand{\etalchar}[1]{$^{#1}$}
\def\cprime{$'$}
\providecommand{\bysame}{\leavevmode\hbox to3em{\hrulefill}\thinspace}
\providecommand{\MR}{\relax\ifhmode\unskip\space\fi MR }
\providecommand{\MRhref}[2]{%
  \href{http://www.ams.org/mathscinet-getitem?mr=#1}{#2}
}
\providecommand{\href}[2]{#2}
\begin{thebibliography}{EPW16}

\bibitem[And18]{Andrews}
Scott Andrews, \emph{The unipotent modules of {${\rm GL}_n(\Bbb F_q)$} via
  tableaux}, J. Algebraic Combin. \textbf{47} (2018), no.~1, 1--15.

\bibitem[BC72]{BensonCurtis}
C.~T. Benson and C.~W. Curtis, \emph{On the degrees and rationality of certain
  characters of finite {C}hevalley groups}, Trans. Amer. Math. Soc.
  \textbf{165} (1972), 251--273.

\bibitem[CEF15]{CEF}
Thomas Church, Jordan~S. Ellenberg, and Benson Farb, \emph{F{I}-modules and
  stability for representations of symmetric groups}, Duke Math. J.
  \textbf{164} (2015), no.~9, 1833--1910.

\bibitem[Cur75]{Curtis75}
Charles~W. Curtis, \emph{Reduction theorems for characters of finite groups of
  {L}ie type}, J. Math. Soc. Japan \textbf{27} (1975), no.~4, 666--688.

\bibitem[DJ04]{DipperJames}
Richard Dipper and Gordon James, \emph{On {S}pecht modules for general linear
  groups}, J. Algebra \textbf{275} (2004), no.~1, 106--142.

\bibitem[Dud18]{Dudas-EMSS}
Olivier Dudas, \emph{Lectures on modular {D}eligne-{L}usztig theory}, Local
  Representation Theory and Simple Groups, Series of Lectures in Mathematics,
  vol.~29, European Mathematical Society, 2018, pp.~107--177.

\bibitem[EPW16]{EPW}
Ben Elias, Nicholas Proudfoot, and Max Wakefield, \emph{The {K}azhdan-{L}usztig
  polynomial of a matroid}, Adv. Math. \textbf{299} (2016), 36--70.

\bibitem[Ged17]{thag}
Katie~R. Gedeon, \emph{Kazhdan-{L}usztig polynomials of thagomizer matroids},
  Electron. J. Combin. \textbf{24} (2017), no.~3, Paper 3.12, 10.

\bibitem[GLX{\etalchar{+}}]{Chinese-uniform}
Alice~L.L. Gao, Linyuan Lu, Matthew~H.Y. Xie, Arthur~L.B. Yang, and Philip~B.
  Zhang, \emph{{The {K}azhdan-{L}usztig polynomials of uniform matroids}},
  \textsf{arXiv:1806.10852}.

\bibitem[GPY17]{GPY}
Katie Gedeon, Nicholas Proudfoot, and Benjamin Young, \emph{The equivariant
  {K}azhdan--{L}usztig polynomial of a matroid}, J. Combin. Theory Ser. A
  \textbf{150} (2017), 267--294.

\bibitem[GW18]{GanW}
Wee~Liang Gan and John Watterlond, \emph{A representation stability theorem for
  {VI}-modules}, Algebr. Represent. Theory \textbf{21} (2018), no.~1, 47--60.

\bibitem[HL83]{Howlett-Lehrer}
R.~B. Howlett and G.~I. Lehrer, \emph{Representations of generic algebras and
  finite groups of {L}ie type}, Trans. Amer. Math. Soc. \textbf{280} (1983),
  no.~2, 753--779.

\bibitem[HRS]{HRS}
Thomas Hameister, Sujit Rao, and Connor Simpson, \emph{Chow rings of vector
  space matroids}, \textsf{arXiv:1802.04241}.

\bibitem[KW]{Karn-Wakefield}
Trevor~K. Karn and Max~D. Wakefield, \emph{{Stirling numbers in braid matroid
  {K}azhdan-{L}usztig polynomials}}, \textsf{arXiv:1802.00849}.

\bibitem[Lus77]{Lusztig-Coxeter}
G.~Lusztig, \emph{Coxeter orbits and eigenspaces of {F}robenius}, Invent. Math.
  \textbf{38} (1976/77), no.~2, 101--159.

\bibitem[LXY]{fan-wheel-whirl}
Linyuan Lu, Matthew~H.Y. Xie, and Arthur~L.B. Yang, \emph{{{K}azhdan-{L}usztig
  polynomials of fan matroids, wheel matroids and whirl matroids}},
  \textsf{arXiv:1802.03711}.

\bibitem[PY17]{fs-braid}
Nicholas Proudfoot and Ben Young, \emph{Configuration spaces, {$\rm
  FS^{op}$}-modules, and {K}azhdan-{L}usztig polynomials of braid matroids},
  New York J. Math. \textbf{23} (2017), 813--832.

\end{thebibliography}
\bibliographystyle{amsalpha}

\end{document}